\RequirePackage[l2tabu, orthodox]{nag}

\documentclass[
11pt,                          
english                        
]{article}

\usepackage[english]{babel}    
\usepackage{amsmath}           
\usepackage[utf8]{inputenc}    
\usepackage[T1]{fontenc}       
\usepackage{longtable}         
\usepackage{exscale}           
\usepackage[final]{graphicx}   
\usepackage[sort]{cite}        
\usepackage{eucal}
\usepackage{array}             
\usepackage{wasysym}           
\usepackage[a4paper]{geometry} 
\usepackage{xspace}            
\usepackage{amssymb}           
\usepackage{bbm}               
\usepackage{mathtools}         
\usepackage[amsmath,thmmarks,hyperref]{ntheorem} 
\usepackage{mathrsfs}          
\usepackage{stmaryrd}          
\usepackage{paralist}          

\usepackage[expansion=false    
           ]{microtype}        
\usepackage[nottoc]{tocbibind} 
\usepackage[
           final=true,         
           pdfpagelabels       
           ]{hyperref}         

\author{
  \textbf{Pierre Bieliavsky}\thanks{\texttt{pierre.bieliavsky@uclouvain.be}}\\[0.3cm]
  Faculté des sciences\\
  Ecole de mathématique (MATH)\\
  Institut de recherche en mathématique et physique (IRMP)\\
  Chemin du Cyclotron 2 bte L7.01.02 \\
  1348 Louvain-la-Neuve\\
  Belgium\\[1cm]
  \textbf{Chiara Esposito}\thanks{\texttt{chiara.esposito@mathematik.uni-wuerzburg.de}},
  \textbf{Stefan Waldmann}\thanks{\texttt{stefan.waldmann@mathematik.uni-wuerzburg.de}},
  \textbf{Thomas Weber}\thanks{\texttt{thomas.weber@stud-mail.uni-wuerzburg.de}}\\[0.3cm]
  Institut für Mathematik \\
  Lehrstuhl für Mathematik X \\
  Universität Würzburg \\
  Campus Hubland Nord \\
  Emil-Fischer-Straße 31 \\
  97074 Würzburg \\
  Germany
}

\renewcommand{\mathbb}[1]{\mathbbm{#1}}

\numberwithin{equation}{section}

\allowdisplaybreaks

\renewcommand{\arraystretch}{1.2}

\let\originalleft\left
\let\originalright\right
\renewcommand{\left}{\mathopen{}\mathclose\bgroup\originalleft}
\renewcommand{\right}{\aftergroup\egroup\originalright}

\theoremheaderfont{\normalfont\bfseries}
\theorembodyfont{\itshape}
\newtheorem{lemma}{Lemma}[section]
\newtheorem{proposition}[lemma]{Proposition}
\newtheorem{theorem}[lemma]{Theorem}
\newtheorem{corollary}[lemma]{Corollary}
\newtheorem{definition}[lemma]{Definition}

\theorembodyfont{\rmfamily}

\newtheorem{example}[lemma]{Example}
\newtheorem{remark}[lemma]{Remark}

\makeatletter
\def\theorem@checkbold{}
\makeatother

\theoremheaderfont{\scshape}
\theorembodyfont{\normalfont}
\theoremstyle{nonumberplain}
\theoremseparator{:}
\theoremsymbol{\hbox{$\boxempty$}}
\newtheorem{proof}{Proof}

\newenvironment{propositionlist}{\begin{compactenum}[\itshape i.)]}{\end{compactenum}}
\newenvironment{definitionlist}{\begin{compactenum}[\itshape i.)]}{\end{compactenum}}

\newcommand{\lie}[1]          {\mathfrak{#1}}

\newcommand{\tensor}[1][{}]           {\mathbin{\otimes_{\scriptscriptstyle{#1}}}}

\newcommand{\algebra}[1]      {\mathscr{#1}}

\newcommand{\acts}            {\mathbin{\triangleright}}

\newcommand{\Cinfty}         {\mathscr{C}^\infty}

\newcommand{\Secinfty}       {\Gamma^\infty}

\DeclarePairedDelimiter{\Schouten}{\llbracket}{\rrbracket}

\newcommand{\argument}       {\,\cdot\,}

\newcommand{\Anti}                    {\Lambda}

\newcommand{\opp}            {\mathrm{opp}}

\newcommand{\D}              {\mathop{}\!\mathrm{d}}
\DeclareMathOperator{\spann} {\mathrm{span}}
\newcommand{\id}             {\mathsf{id}}

\title{Obstructions for Twist Star Products}

\date{July 2016}

\begin{document}

%
%

\maketitle

%
%

\begin{abstract}
    In this short note we point out that not every star product is
    induced by a Drinfel'd twist by showing that not every Poisson
    structure is induced by a classical $r$-matrix. Examples include
    the higher genus symplectic Pretzel surfaces and the symplectic
    sphere $\mathbb{S}^2$.
\end{abstract}

\newpage

%
%

\tableofcontents

%
%

\section{Introduction}
\label{sec:Introduction}

Deformation quantization as introduced in \cite{bayen.et.al:1978a} has
found many applications in mathematical physics beyond the original
scope of quantizing classical mechanical systems. In particular, the
formal star products provide examples of noncommutative manifolds as
needed in noncommutative geometry \cite{connes:1994a}. Here one finds
many approaches to fundamental questions of the physical nature of the
geometry of spacetime.

While for generic star products certain desirable constructions will
not be possible, there are particular classes of star products with
better behaviour: star products induced by a Drinfel'd twist by means
of universal deformation formulas. Recall that for a formal Drinfel'd
twist
$\mathcal{F} \in (\mathcal{U}(\lie{g}) \tensor
\mathcal{U}(\lie{g}))[[\hbar]]$
one can deform every algebra $\algebra{A}$ on which the Lie algebra
$\lie{g}$ acts by derivations. Indeed, for
$a, b \in \algebra{A}[[\hbar]]$ we can define the product
\begin{equation}
    \label{eq:TheUDF}
    a \star_{\mathcal{F}} b
    =
    m (\mathcal{F}^{-1} \acts (a \tensor b)),
\end{equation}
where
$m\colon (\algebra{A} \tensor \algebra{A})[[\hbar]] \longrightarrow
\algebra{A}[[\hbar]]$
is the undeformed multiplication of $\algebra{A}$, extended
$\hbar$-bilinearly, and $\acts$ denotes the Lie algebra action
extended to an action of the universal enveloping algebra as usual.
The conditions on $\mathcal{F}$ guarantee that $\star_{\mathcal{F}}$
is again an associative product whenever $m$ was associative.

Beside this simple and universal formula, deformations by means of a
twist allow for many additional constructions. In particular, every
module over the undeformed algebra on which the Lie algebra acts as
well, can be deformed into a module over the deformed algebra. In
noncommutative geometry this aspect is of great importance: here we
consider as algebra to start with the functions
$\algebra{A} = \Cinfty(M)$ on a manifold $M$. Then a Lie algebra
action on $\algebra{A}$ is the same thing as a Lie algebra action by
vector fields on $M$. Hence this action lifts to all tensor bundles
over $M$. Therefore we can deform the sections of all tensor bundles
over $M$ in a coherent way. For a general star product this would not
be possible in such a nice way: instead, the sections of every vector
bundle have to be deformed individually but with no significant
relations between different bundles.

While providing nice additional features, the question has to be
raised whether such star products occur rarely or often: one has of
course many examples at hand but it is not clear whether there are
hard obstructions for a given star product to be induced by a
twist. Moreover, recall that a star product induces as semiclassical
limit on $M$ a Poisson structure. Thus a natural question is whether a
given Poisson structure allows for a star product induced by a twist.

In this work we provide several obstructions for this to be
possible. First, we have some quite general obstructions in the case
when the Poisson bracket is symplectic:
\begin{theorem}
    \label{theorem:TheReallyCoolTheorem}%
    Given a connected and compact symplectic manifold $(M, \omega)$
    endowed with a star product induced by a twist of some Lie
    algebra. Then $M$ is homogeneous and there exists a non-degenerate
    $r$-matrix for a (possibly different) Lie algebra whose Lie group
    acts transitively and effectively on $M$.
\end{theorem}
Using this theorem, we show that the general obstructions can be
realized in many explicit examples like for the symplectic sphere
$\mathbb{S}^2$ or all higher genus Pretzel surfaces.

The strategy is to show first that the first order of the twist, the
classical $r$-matrix associated to the twist, induces the Poisson
structure coming from the star product $\star_{\mathcal{F}}$: this is
the semiclassical version of \eqref{eq:TheUDF}. In a second step, one
uses the fact that $r$ is contained in a Lie subalgebra in such a way
that $r$ becomes non-degenerate there. Third, we show that in the
symplectic case the action of this subalgebra is infinitesimally
transitive. If it is integrable, say for $M$ compact, and if $M$ is
connected, then we obtain a transitive action. Finally, a last
argument shows that passing to the quotient by the kernel of the
action does not spoil this picture.

Counterexamples are then easily obtained by either symplectic
manifolds which are not homogeneous at all, like the higher genus
Pretzel surfaces, or by homogeneous manifolds where none of the acting
Lie group admits a non-degenerate $r$-matrix: this will be the case
for the sphere $\mathbb{S}^2$.

As a last remark we note that there is also a slightly weaker concept
than star products induced by a twist: we only require that the star
product allows for a braiding structure encoded by a universal
$R$-matrix $R$. It turns out that also in this case, we arrive at the
same semiclassical obstructions and hence the above statements gives
also obstructions for such braiding structures.

This paper is partially based on the master thesis \cite{weber:2016a}.

\noindent

\textbf{Acknowledgements:} We would like to thank Martin Bordemann,
Alexander Schenkel and Jonas Schnitzer for valuable discussions and
useful suggestions.

%
%

\section{Preliminaries}

Let $\lie{g}$ be a finite-dimensional real Lie algebra and consider
the algebra $\mathcal{U}(\lie{g})[[\hbar]]$ of formal power series
with coefficients in the universal enveloping algebra
$\mathcal{U}(\lie{g})$.  It is known that it can be easily endowed
with a topologically free Hopf algebra structure, denoted by
$(\mathcal{U}(\lie{g})[[\hbar]], \Delta, \epsilon, S)$. To fix our
notation and sign conventions, we recall the following two definitions
of a Drinfel'd twist and a (classical) $r$-matrix, see
\cite{drinfeld:1983a, drinfeld:1988a}.
\begin{definition}[Drinfel'd twist]
    \label{def:TwistUEA}%
    An element
    $\mathcal{F}\in(\mathcal{U}(\lie{g}) \tensor
    \mathcal{U}(\lie{g}))[[\hbar]]$
    is said to be a twist on $\mathcal{U}(\lie{g})[[\hbar]]$ if the
    following three conditions are satisfied.
    \begin{definitionlist}
    \item
        $\mathcal{F} = 1 \tensor 1 + \sum_{k=1}^\infty \hbar^k
        \mathcal{F}_k$.
    \item
        $(\mathcal{F}\tensor 1)(\Delta\tensor 1)(\mathcal{F}) = (1
        \tensor \mathcal{F})(1\tensor\Delta)(\mathcal{F})$.
    \item
        $(\epsilon\tensor 1)\mathcal{F} =
        (1\tensor\epsilon)\mathcal{F} = 1$.
  \end{definitionlist}
\end{definition}
Let $\Schouten{\argument, \argument}$ be the unique extension of the
Lie bracket to $\Anti^\bullet \lie{g}$ turning the Grassmann algebra
into a Gerstenhaber algebra.
\begin{definition}[$r$-Matrix]
    \label{def:TriangularLieAlgebra}%
    An element $r\in \lie{g} \wedge \lie{g}$ is said to be a classical
    $r$-matrix if it satisfies the classical Yang-Baxter equation
    $\Schouten{r, r} = 0$.
\end{definition}
The relation between the two concepts is well-known to be as follows,
see \cite{drinfeld:1988a} or \cite[Thm.~1.14]{giaquinto.zhang:1998a}:
\begin{proposition}
    \label{prop:TwistR}%
    Let $\mathcal{F}$ be a twist on
    $\mathcal{U}(\lie{g})[[\hbar]]$. Then the antisymmetric part of
    its first order
    \begin{equation}
        \label{eq:TwistR}
        r
        =
        \mathcal{F}_{1}^{-1} - \sigma(\mathcal{F}_{1}^{-1})
        =
        - \mathcal{F}_1 + \sigma (\mathcal{F}_1)
        \in
        \mathcal{U}(\lie{g})\wedge \mathcal{U}(\lie{g})
    \end{equation}
    is a classical $r$-matrix $r \in \lie{g} \wedge \lie{g}$. Here
    $\sigma$ denotes the flip isomorphism.
\end{proposition}

It is important to recall that there is a one-to-one correspondence
between non-degenerate $r$-matrices and symplectic structures on
$\lie{g}$, see \cite[Prop.~3.3]{etingof.schiffmann:1998a}. For this
reason, we are interested to discuss, given a Lie algebra $\lie{g}$,
under which conditions it is possible to endow it with a
non-degenerate $r$-matrix. It turns out that if $\lie{g}$ is
semisimple, this is not possible, see
\cite[Prop.~5.2]{etingof.schiffmann:1998a} for the complex and simple
case, which clearly extends to real and semisimple Lie algebras:
\begin{proposition}
    \label{prop:NoRMatrix}%
    Let $ \lie{g} $ be a semisimple Lie algebra. Then there are no
    non-degenerate $r$-matrices on it.
\end{proposition}

On the other hand, one can always find a suitable (finite-dimensional)
Lie subalgebra $\lie{g}_{r} \subseteq \lie{g} $ such that an
$r$-matrix on $ \lie{g} $ is non-degenerate if viewed as an element of
$\lie{g}_{r}\wedge \lie{g}_{r}$, see
\cite[Sect.~3.5]{etingof.schiffmann:1998a}:
\begin{proposition}
    \label{prop:EtingofSubalg}%
    Let $r \in \lie{g} \wedge \lie{g} $ be an $r$-matrix for a real
    Lie algebra $\lie{g}$.
    \begin{propositionlist}
    \item \label{item:ESsubalgebra} The subspace $\lie{g}_{r}$ defined
        by
        \begin{equation}
            \label{eq:Etingof}
            \lie{g}_{r}
            =
            \left\{
                (f\tensor\mathbb{1})r \in \lie g
                \; \big| \;
                f \in \lie{g}^{*}
            \right\}
        \end{equation}
        is a finite-dimensional Lie subalgebra of $\lie{g}$.
    \item \label{item:rInhrwedgehr} We have
        $r \in \lie{g}_r \wedge \lie{g}_r$.
    \item \label{item:rNondegInhr} $r$ is non-degenerate, if viewed as
        en element $r \in \lie{g}_r \wedge \lie{g}_r$.
    \end{propositionlist}
\end{proposition}
\begin{definition}[Symplectic subalgebra]
    \label{definition:ESsubalgebra}%
    The Lie subalgebra $\lie{g}_{r} \subseteq \lie{g}$ defined in
    \eqref{eq:Etingof} is called the symplectic subalgebra of the
    $r$-matrix $r \in \lie{g} \wedge \lie{g}$.
\end{definition}

In the case of a real Lie algebra $\lie{g}$, the corresponding Lie
group is denoted by $G_r$.

As already mentioned, given a twist on the universal enveloping
algebra, we can easily define an associative star product on any
$\mathcal{U}(\lie{g})$-module algebra. In particular, let us consider
the algebra $\Cinfty(M)$ of smooth functions on a manifold $M$ with
pointwise multiplication $m$ and assume that it is a left
$\mathcal{U}(\lie{g})$-module algebra. In other words, let us consider
a Hopf algebra action
\begin{equation}
    \label{eq:Action}
    \acts
    \colon
    \mathcal{U}(\lie{g}) \times \Cinfty(M)
    \longrightarrow
    \Cinfty(M),
\end{equation}
which makes $\Cinfty(M)$ into a left $\mathcal{U}(\lie{g})$-module
algebra. Since $\lie{g} \subseteq \mathcal{U}(\lie{g})$ are the
primitive elements, the Lie algebra elements act as derivations of
$\Cinfty(M)$, i.e. as vector fields on $M$. This defines a Lie algebra
action $\phi$ of $\lie{g}$ on $M$. Since the elements of $\lie{g}$
generate $\mathcal{U}(\lie{g})$, the action $\acts$ is given by
differential operators: more precisely, using the natural filtration
of the universal enveloping algebra, the order of the differential
operator $X \acts \argument$ is at most $k \in \mathbb{N}$ when $X \in
\mathcal{U}^{(k)}(\lie{g})$.  From now on we also denote by $\acts$
the extension of the action to formal power series $\acts \colon
\mathcal{U}(\lie{g})[[\hbar]] \times \Cinfty(M)[[\hbar]]
\longrightarrow \Cinfty(M)[[\hbar]]$.  Conversely, every Lie algebra
action $\phi$ of $\lie{g}$ on $M$ determines via the fundamental
vector fields $\phi(\xi) \in \Secinfty(TM)$ a representation of
$\lie{g}$ on $\Cinfty(M)$ by derivations which therefore extends to a
Hopf algebra action $\acts$ as above.

Since the first-order commutator of an associative commutative algebra
is necessarily a Poisson bracket on the (un-deformed) algebra, see
e.g.~\cite[Proposition~6.2.24]{waldmann:2007a}, we have the following
statement:
\begin{lemma}
    \label{lem:StarTwist}%
    The product defined by
    \begin{align}
        \label{eq:StarTwist}
        f \star_\mathcal{F} g
        =
        m(\mathcal{F}^{-1}\acts(f\tensor g))
    \end{align}
    for $f, g \in \Cinfty(M)[[\hbar]]$ is an associative star product
    quantizing the Poisson structure
    \begin{equation}
        \label{eq:ThePoissonStructure}
        \{f, g\}
        =
        m (r \acts (f \tensor g)),
    \end{equation}
    where $r$ is the $r$-matrix associated to the twist $\mathcal{F}$.
\end{lemma}
\begin{remark}[Taking the classical limit]
    \label{remark:ImportantClassicalLimitRemark}%
    We can summarize this discussion now as follows. Taking the
    classical limit of a twist $\mathcal{F}$ gives a classical
    $r$-matrix $r$. Taking the classical limit of a Hopf algebra
    action $\acts$ of $\mathcal{U}(\lie{g})[[\hbar]]$ and restricting
    it to the Lie algebra gives a Lie algebra action $\phi$ on $M$ by
    vector fields, and taking the classical limit of a star product
    gives a Poisson bracket. If now the star product comes from a
    twist via some Hopf algebra action then the \emph{corresponding}
    Poisson bracket comes from the \emph{corresponding} $r$-matrix and
    the \emph{corresponding} Lie algebra action.
\end{remark}

Finally, we will also need the quantum analogue of $r$-matrices:
\begin{definition}[Universal $R$-matrix]
    An invertible element $R \in (\mathcal{U}(\lie{g}) \tensor
    \mathcal{U}(\lie{g}))[[\hbar]]$ is called a formal universal
    $R$-matrix if
    \begin{definitionlist}
    \item $R = 1 \tensor 1 + \cdots$,
    \item $\Delta^\opp = R \Delta R^{-1}$,
    \item
        $(1 \tensor \Delta) R = (R_1\tensor 1 \tensor R_2)(R \tensor
        1)$
        and
        $(\Delta \tensor 1)R = (R_1\tensor 1 \tensor R_2)(1 \tensor
        R)$, using Sweedler's notation $R = R_1 \tensor R_2$.
    \end{definitionlist}
\end{definition}
It is known that $R$ is the quantization of the $r$-matrix over
$\lie{g}$. In other words, given a formal $R$-matrix $R$ for
$\lie{g}$, we can define an $r$-matrix for $\lie g$ by the first order
term
\begin{equation}
    \label{eq:RmatrixFirstOrder}
    R = 1 \tensor 1 + \hbar r + \cdots.
\end{equation}

A star product $\star$ is called quasi-commutative with respect to $R$
if one has an action of $\lie{g}$ by derivations such that the
extended actions $\acts$ satisfies
\begin{equation}
    \label{eq:fgQuasiCommutative}
    f \star g =
    (R_2 \acts g) \star (R_1 \acts f)
\end{equation}
for all $f, g \in \Cinfty(M)[[\hbar]]$, see
\cite[Def.~5.8]{aschieri.schenkel:2014a}. A simple verification leads
to the following statement:
\begin{proposition}
    \label{proposition:QuasiCommutativeClassicalLimit}%
    Suppose $\star$ is a quasi-commutative star product on $M$ with
    respect to a universal $R$-matrix $R$ for $\lie{g}$. Then the
    corresponding Poisson structure of $\star$ is induced by the
    corresponding $r$-matrix $r$.
\end{proposition}

%
%

\section{Drinfel'd twist and transitive actions}
\label{sec:DrinfeldTwistTransitiveAction}

In this section we prove the main result of this paper, which allows
us to understand the conditions under which a star product on a
manifold $M$ can be induced by a twist on the universal enveloping
algebra of a Lie algebra $\lie{g}$ acting on $M$.

Let $r$ be an $r$-matrix on $\lie{g}$.  For a finite-dimensional Lie
algebra $\lie{g}$ with basis $\{e_1, \ldots , e_n\}$, we have
\begin{equation}
    \label{eq:LocalR}
    r
    =
    \frac{1}{2}\sum_{i,j = 1}^n r^{ij} e_i \wedge e_j.
\end{equation}
Given a Lie algebra action
$\phi\colon \lie{g} \to \Gamma^{\infty}(TM) $ we can define a
Poisson structure on $M$ as the image via $\phi$ of the $r$-matrix,
i.e.
\begin{align}
    \label{eq:PiR}
    \pi_{p}
    =
    \frac{1}{2}\sum_{i,j=1}^{n}r^{ij}\phi(e_{i})_{p}\wedge\phi(e_{j})_{p},
\end{align}
for $p \in M$.  If $\pi$ is non-degenerate, it induces a symplectic
structure via $\omega(X_f, X_g) = \pi (\D f, \D g)$, where $X_f$
denotes the Hamiltonian vector field of $f \in \Cinfty(M)$ with
respect to $\pi$.

In order to prove our main result we need some preparation. First we
have to pass from a Lie algebra action to a Lie group action of an
integrating Lie group $G$ of $\lie{g}$.  For the moment we have to
assume that this is the case.
\begin{lemma}
    \label{lem:LocTransitive}%
    Let $ (M, \omega) $ be a symplectic manifold, $r \in \lie{g}
    \wedge \lie{g}$ an $r$-matrix and let $\phi$ be a Lie algebra
    action of $\lie{g}$ on $M$ integrating to a Lie group action
    $\Phi\colon G \times M \rightarrow M $ such that $\omega$ is the
    image of $r$ via \eqref{eq:PiR}. Then $\Phi$ is locally
    transitive.  Moreover, the restriction $\Phi|_{G_r}$ of $\Phi$ to
    the Lie subgroup $G_r$ corresponding to $r$ is locally transitive.
\end{lemma}
\begin{proof}
    It suffices to show the second statement.  We observe that we can
    find a basis $\{e_1, \ldots, e_n\}$ of $\lie{g}$ such that there
    is an integer $k \leq n$ such that $\{e_1, \ldots, e_k\}$ is a
    basis of the symplectic Lie subalgebra $\lie{g}_r$ determined by
    $r$ and
    \begin{equation}
        r
        =
        \frac{1}{2}\sum_{i = 1}^{k} r^{ij} e_i \wedge e_j.
    \end{equation}
    Now consider an arbitrary vector $v_p \in T_pM$, for $p\in M$.
    Since for a symplectic Poisson tensor, the map
    $\pi^\sharp \colon T_p^*M \ni \alpha_{p} \mapsto \pi_p (\argument,
    \alpha_p) \in T_pM$
    is surjective, there exists an element $\alpha_p \in T_p^*M $ such
    that
    \begin{equation}
        v_p
        =
        \pi_p(\alpha_{p}, \argument)
        =
        (\alpha_p \tensor \mathbb{1}) \pi_p
        =
        \sum_{i,j = 1}^n
        r^{ij}  \alpha_p (\phi(e_i)_p) \phi(e_j)_p
        =
        \sum_{i,j = 1}^k
        r^{ij}  \alpha_p (\phi(e_i)_p) \phi(e_j)_p.
    \end{equation}
    In other words,
    \begin{equation}
        v_{p}
        \in
        \spann_{\mathbb{R}}\left\{
            \sum_{j=1}^k r^{1j} \phi(e_j)_{p},
            \ldots,
            \sum_{j=1}^k r^{kj} \phi(e_j)_{p}
        \right\}
        \subseteq
        \spann_{\mathbb{k}}\left\{
            \phi(e_1)_p, \ldots, \phi(e_k)_p
        \right\}.
    \end{equation}
    Thus, the map $\phi|_p\colon \lie{g}_r \longrightarrow T_pM$ is
    surjective and hence the action of $G_r$ is locally transitive.
\end{proof}

A locally transitive action is known to be transitive if the
underlying manifold $M$ is connected, and hence a homogeneous space:
\begin{lemma}
    \label{lem:Homogeneous}%
    Let $ \Phi\colon G \times M \longrightarrow M $ be a locally
    transitive Lie group action on a connected manifold $M$. Then
    there is only one orbit of $\Phi$ and it coincides with $M$.
\end{lemma}

Putting these results together, we can prove the main theorem of this
paper:
\begin{theorem}
    \label{thm:Main}
    Let $ (M, \omega) $ be a connected symplectic manifold, $r \in
    \lie{g} \wedge \lie{g} $ an $r$-matrix and $\phi$ a Lie algebra
    action of $\lie{g}$ on $ M $ that integrates to a Lie group action
    $\Phi\colon G \times M \longrightarrow M$ such that, for any $p
    \in M$, the symplectic bivector field $\pi$ corresponding to
    $\omega$ is given by \eqref{eq:PiR}.  Then $M$ can be structured
    as a homogeneous space for the Lie subgroup $G_r$ corresponding to
    the symplectic Lie algebra $\lie{g}_r$ with respect to $r$.
\end{theorem}
\begin{proof}
    This is a direct consequence of Lemma~\ref{lem:LocTransitive} and
    Lemma~\ref{lem:Homogeneous}.
\end{proof}

The first important consequence of this statement is that under the
assumption that the Lie algebra action integrates, we have an
obstruction for star products to be induced by twists. As before, we
denote by $\lie{g}_r$ the symplectic Lie subalgebra with respect to
the $r$-matrix $r$ corresponding to the twist $\mathcal{F}$ and $G_r$
is the connected and simply-connected Lie group integrating it.
\begin{corollary}
  \label{corollary:TwistHomogeneous}%
  Let $ (M, \omega)$ be a connected symplectic manifold.  Assume that
  there is a star product $\star_{\mathcal{F}}$ induced by a twist via
  the action $\acts$ such that the corresponding Lie algebra action
  integrates to a Lie group action. Then $M$ is a homogeneous
  $G_r$-space.
\end{corollary}
\begin{corollary}
    \label{cor:TwistHomogeneous}%
    Let $ (M, \omega)$ be a connected and compact symplectic manifold
    with a star product $\star_{\mathcal{F}}$ induced by a twist.
    Then $M$ is a homogeneous $G_r$-space.
\end{corollary}
\begin{proof}
    If we assume compactness then any Lie algebra action comes from a
    Lie group action by Palais' theorem \cite{palais:1957a} since the
    fundamental vector fields are necessarily complete.
\end{proof}

We can even say something more about the homogeneous space. There is
of course a certain redundancy in the Lie group at the moment: the
action could have a large kernel. The following statement shows that
we can pass from a transitive action to a transitive and effective
action without loosing the $r$-matrix:
\begin{proposition}
    \label{prop:TransitiveEffective}%
    Let $\Phi\colon G \times M \longrightarrow M$ be a smooth Lie
    group action on a manifold $M$ with kernel
    \begin{equation}
        \ker \Phi
        =
        \{g \in G \; | \; \Phi_g = \id_M \},
    \end{equation}
    and let $\ker \phi \subseteq \lie{g}$ be the corresponding Lie
    ideal.
    \begin{propositionlist}
    \item \label{item:PreserveTransGetEff} If the action $ \Phi$ is
        transitive, the induced action
        \begin{equation}
            \label{QuotientAction}
            \Psi\colon
            G \big/\ker \Phi \times M \ni ([g], x)
            \; \mapsto \;
            \Phi(g, x) \in M
        \end{equation}
        of the quotient Lie group $G/\ker \Phi$ on $M$ is effective
        and still transitive.
    \item For a classical $r$-matrix $r \in \Anti^2\lie{g}$ the image
        $[r] \in\Anti^2 \lie{g}\big/ \ker \phi$ is a classical
        $r$-matrix for $\lie{g} \big/ \ker\phi$.
    \item The induced Poisson structures of $r$ and $[r]$ on $M$
        coincide.
    \end{propositionlist}
\end{proposition}
\begin{proof}
    The first part is clear and so is the second. Since the Poisson
    structure on $M$ is obtained by applying $\phi \wedge \phi$ to
    $r$, the contributions in the kernel $\ker \phi$ will not
    contribute and hence $[r]$ yields the same Poisson structure.
\end{proof}
\begin{corollary}
    \label{corollary:NicyfyTheAction}%
    Let $(M, \omega)$ be a compact symplectic manifold such that the
    Poisson structure is induced by an $r$-matrix. Then there is a
    non-degenerate $r$-matrix $r \in \Anti^2 \lie{g}$ in a Lie algebra
    $\lie{g}$ with corresponding Lie group $G$ acting transitively and
    effectively on $M$.
\end{corollary}
\begin{proof}
    By Theorem~\ref{thm:Main} we can assume that there is an
    $r$-matrix $r \in \Anti^2 \tilde{\lie{g}}$ in some Lie algebra
    $\tilde{\lie{g}}$ such that the corresponding group $\tilde{G}$
    acts transitively on $M$ via $\Phi$. Then we can quotient by the
    kernel of the action $\Phi$ to obtain a transitive and effective
    action of $G = \tilde{G} \big/ \ker \Phi$, still having a
    classical $r$-matrix inducing the same Poisson bracket. Finally,
    we can pass to the symplectic Lie subalgebra $\lie{g}_r$ and take
    the Lie subgroup $G_r$ of $G$ which still acts transitively and
    effectively.
\end{proof}

%
%

\subsection{Counterexamples: Pretzel surfaces}

The above results provide a large class of examples in which star
products can not be induced via a twist.  The first class of examples
that we discuss here is given by manifolds which can not be written as
homogeneous spaces at all. For this reason, it is important to recall
some useful characterization of homogeneous spaces in terms of Euler
characteristic $\chi(M)$.  The following classical theorem of Mostow
gives now a necessary condition for a manifold to be homogeneous
\cite{mostow:2005a}, see \cite{mostow:1950a} for the case of surfaces:
\begin{theorem}[Mostow]
    \label{thm:Mostow}%
    Let $M = G/H$ be a connected compact homogeneous space. Then
    $\chi(M) \geq 0$.
\end{theorem}

From this we immediately get the following class of examples of
symplectic manifolds which do not allow a star product induced by a
twist:
\begin{example}[Pretzels]
    Let us consider the Pretzel surfaces $T(g)$ of genus
    $g\in\mathbb{N}_{0}$, which are compact and connected symplectic
    manifolds with Euler characteristic
    \begin{equation}
        \chi(T(g))
        =
        2-2g.
    \end{equation}
    For $g > 1$ the Euler characteristic is negative. Thus, according
    to Theorem~\ref{thm:Mostow} the Pretzel surfaces $T(g)$ are not
    homogeneous spaces for $g > 1$ at all.  As an immediate
    consequence of Corollary~\ref{cor:TwistHomogeneous}, no symplectic
    star product on $T(g)$ for $g > 1$ can be induced by a twist.
\end{example}

By the same line of argument, one can of course also construct
examples of compact symplectic manifolds in higher dimensions.

%
%

\subsection{Counterexample: Symplectic Sphere}

While Mostow's Theorem provides an easy way to rule out the higher
genus Pretzel surfaces, the situation for the symplectic sphere
$\mathbb{S}^2$ and the symplectic torus $\mathbb{T}^2$ is more
involved: for the torus one can use e.g. the canonical action of
$\mathbb{T}^2$ on itself. Then the abelian twist provides a star
product, the Weyl-Moyal star product, on the torus. So here we do not
have any obstructions.

Also the symplectic sphere is a very well-known example of a
homogeneous space.  Nevertheless, we can show that also in this case
no star product on the symplectic sphere can be induced by a
twist. The main idea is that even though the sphere is a homogeneous
space in several ways, none of them allows for a twist: one uses the
well-known classification of transitive and effective actions on
$\mathbb{S}^2$ and shows by hand that all of them come from semisimple
Lie groups.  The classification results we use are due to Onishchik
\cite{Onishchik1967, Onishchik1968, Onishchik1993} as well as earlier
work by Montgomery and Samelson \cite{Montgomery1942}:
\begin{theorem}[Onishchik]
    \label{theorem:SphereSemiSimple}%
    Any connected Lie group that acts transitively and locally
    effectively on $ \mathbb{S}^{2} $ is semisimple.
\end{theorem}

Using this classification result together with
Corollary~\ref{corollary:NicyfyTheAction} and
Proposition~\ref{prop:NoRMatrix} we arrive at the following statement
for the symplectic sphere:
\begin{proposition}
    \label{proposition:TheSphere}%
    A symplectic Poisson structure on $\mathbb{S}^2$ can not be
    induced by a classical $r$-matrix.
\end{proposition}
\begin{corollary}
    \label{corollary:NoTwistStarOnSphere}%
    There is no star product on the symplectic sphere $\mathbb{S}^{2}$
    induced by a twist.
\end{corollary}

%
%

{
  \footnotesize
  \renewcommand{\arraystretch}{0.5}

}

%
%


%
%

\end{document}